\documentclass{amsproc}

\usepackage[hiresbb]{graphicx}
\usepackage{xcolor}

\usepackage{amsmath,amsthm,amssymb,mathtools}
\allowdisplaybreaks

\usepackage[a4paper,margin=28mm]{geometry}

\usepackage{here}
\usepackage{hyperref}
\usepackage[abbrev]{amsrefs}
\renewcommand{\MR}[1]{}

\usepackage{setspace}
\usepackage{aliascnt}
\usepackage[capitalize,nameinlink,noabbrev,nosort]{cleveref}
\hypersetup{
  colorlinks=true,
  linkcolor=brown,
  citecolor=brown,
  filecolor=brown,
  urlcolor=brown,
  pdftitle={Distribution of Kaneko's val function},
  pdfauthor={Toshiki Matsusaka},
  pdfsubject={2020 Mathematics Subject Classification: 11F03, 37A44, 11A55},
  pdfkeywords={cycle integrals, val function, closed geodesics, continued fractions, ergodic theory},
}

\makeatletter
\@namedef{subjclassname@2020}{\textup{2020} Mathematics Subject Classification}
\makeatother

\theoremstyle{remark}

\newaliascnt{remark}{theoremcounter}
\newtheorem{remark}[remark]{Remark}
\aliascntresetthe{remark}

\theoremstyle{definition}

\newaliascnt{definition}{theoremcounter}
\newtheorem{definition}[definition]{Definition}
\aliascntresetthe{definition}

\newaliascnt{example}{theoremcounter}

\aliascntresetthe{example}

\theoremstyle{plain}

\newaliascnt{lemma}{theoremcounter}
\newtheorem{lemma}[lemma]{Lemma}
\aliascntresetthe{lemma}

\newaliascnt{proposition}{theoremcounter}
\newtheorem{proposition}[proposition]{Proposition}
\aliascntresetthe{proposition}

\newaliascnt{corollary}{theoremcounter}

\aliascntresetthe{corollary}

\newaliascnt{conjecture}{theoremcounter}
\newtheorem{conjecture}[conjecture]{Conjecture}
\aliascntresetthe{conjecture}

\newaliascnt{theorem}{theoremcounter}
\newtheorem{theorem}[theorem]{Theorem}
\aliascntresetthe{theorem}

\newaliascnt{question}{theoremcounter}

\aliascntresetthe{question}

\newaliascnt{exercise}{theoremcounter}

\aliascntresetthe{exercise}

\numberwithin{equation}{section}

\newcommand{\N}{\mathbb{N}}
\newcommand{\Z}{\mathbb{Z}}
\newcommand{\Q}{\mathbb{Q}}
\newcommand{\R}{\mathbb{R}}
\providecommand{\C}{}
\renewcommand{\C}{\mathbb{C}}

\newcommand{\dd}{\mathrm{d}}
\newcommand{\bbH}{\mathbb{H}}

\ExplSyntaxOn
\clist_map_inline:nn
  { A,B,C,D,E,F,G,H,I,J,K,L,M,N,O,P,Q,R,S,T,U,V,W,X,Y,Z }
  {
    \cs_new:cpn { cal#1 } { \mathcal{#1} }
  }
\ExplSyntaxOff

\ExplSyntaxOn
\clist_map_inline:nn
  { A,B,C,D,E,F,G,H,I,J,K,L,M,N,O,P,Q,R,S,T,U,V,W,X,Y,Z }
  {
    \cs_new:cpn { mat#1 } { \mathrm{#1} }
  }
\ExplSyntaxOff

\DeclareMathOperator{\ImNew}{Im}
\renewcommand{\Im}{\ImNew}
\DeclareMathOperator{\ReNew}{Re}
\renewcommand{\Re}{\ReNew}

\newcommand{\id}{\mathrm{id}}
\newcommand{\SL}{\mathrm{SL}}
\newcommand{\PSL}{\mathrm{PSL}}
\newcommand{\PGL}{\mathrm{PGL}}

\DeclareMathOperator{\tr}{tr}
\DeclareMathOperator{\val}{val}
\DeclareMathOperator*{\susp}{susp}

\newcommand{\pmat}[1]{\begin{pmatrix}#1\end{pmatrix}}

\newcommand{\smat}[1]{\bigl(\begin{smallmatrix}#1\end{smallmatrix}\bigr)}

\begin{document}

\title[]{Distribution of Kaneko's val function}

\author[]{Toshiki Matsusaka}
\address{Faculty of Mathematics, Kyushu University, Motooka 744, Nishi-ku, Fukuoka 819-0395, Japan}
\email{matsusaka@math.kyushu-u.ac.jp}


\keywords{cycle integrals, val function, closed geodesics, continued fractions, ergodic theory}
\subjclass[2020]{Primary 11F03, Secondary 37A44, 11A55}



\begin{abstract}
	Kaneko's val function is defined as the normalized cycle integral of the elliptic modular $j$-function along closed geodesics on the modular surface. We prove that, when primitive hyperbolic conjugacy classes are ordered by geodesic length, its values concentrate at the single point 720. More generally, an analogous concentration result holds for every weakly holomorphic modular function $f$ of weight 0, with the concentration point given by Atkin's inner product $(f, 1)_{\mathrm{At}}$. The proof combines an equidistribution theorem following Pollicott with the ergodicity of a continued-fraction suspension flow and uses the decomposition formula of Bengoechea--Imamo\={g}lu to construct a bounded continuous observable.
\end{abstract}

\maketitle



\section{Introduction}

Inspired by Hecke's classical work, Kaneko~\cite{Kaneko2009} introduced the \emph{val function} as the constant term in the hyperbolic Fourier expansion of the elliptic modular $j$-function at a real quadratic irrationality. Equivalently, it is the average of $j$ along the closed geodesic $C_\gamma$ on $\Gamma\backslash\bbH$ associated with a hyperbolic element $\gamma\in\Gamma\coloneqq\PSL_2(\Z)$. For a $\Gamma$-invariant holomorphic function $f\colon\bbH\to\C$, set
\[
	I_f(\gamma)\coloneqq\int_{C_\gamma}f(\tau) \frac{|\dd \tau|}{\Im \tau}
\]
and $\ell(C_\gamma)\coloneqq I_1(\gamma)$. Thus, for
\[
	j(\tau)=q^{-1}+744+196884q+21493760q^2+\cdots,
	\qquad q\coloneqq e^{2\pi i\tau},
\]
Kaneko's function is $\val(\gamma)\coloneqq I_j(\gamma)/I_1(\gamma)$.

Kaneko established basic properties of the val function and made several observations based on extensive numerical experiments. For $\gamma_\phi=\smat{2&1\\1&1}$, whose attracting fixed point is the golden ratio $\phi\coloneqq(1+\sqrt5)/2$, he found numerically that $\val(\gamma_\phi)=706.324813540\ldots$. Kaneko asked the following three questions.
\begin{enumerate}
	\item[(i)] Is $\Re\val(\gamma)\in[\val(\gamma_\phi),744]$ for every hyperbolic $\gamma\in\Gamma$, with both bounds optimal?
	\item[(ii)] How is the val function related to Diophantine approximation, and how does it reflect the Diophantine properties of the fixed points of $\gamma$?
	\item[(iii)] Is $\Im\val(\gamma) \in (-1,1)$, and more generally, how are the imaginary parts distributed over hyperbolic conjugacy classes?
\end{enumerate}
He sought rigorous formulations and proofs of these observations and an arithmetic interpretation of the val function.

Regarding (i) and (ii), Bengoechea--Imamo\={g}lu~\cites{BI2019,BI2020} established continuity properties and explicit bounds. Murakami~\cites{Murakami2021, Murakami2023} obtained related results. In particular, Bengoechea--Imamo\={g}lu proved $\Re\val(\gamma)\leq744$ for every hyperbolic $\gamma\in\Gamma$. The optimality of $744$ had previously been established by P\"{a}pcke~\cite{Papcke2016}, and Bengoechea--Herrero--Imamo\={g}lu~\cites{BHI2026,BHI2026pre} proved the complementary optimal lower bound $\Re\val(\gamma)\geq\val(\gamma_\phi)$, thereby settling (i).

The conjectural bound $-1<\Im\val(\gamma)<1$ in problem~(iii), however, remains open, as does the distribution problem posed there. In this paper, we solve the latter when primitive hyperbolic conjugacy classes are ordered by geodesic length.

\cref{fig:val-distr} plots the values $\val(\gamma)$ in the complex plane.

\begin{figure}[H]
	\centering
	\includegraphics[width=12.5cm]{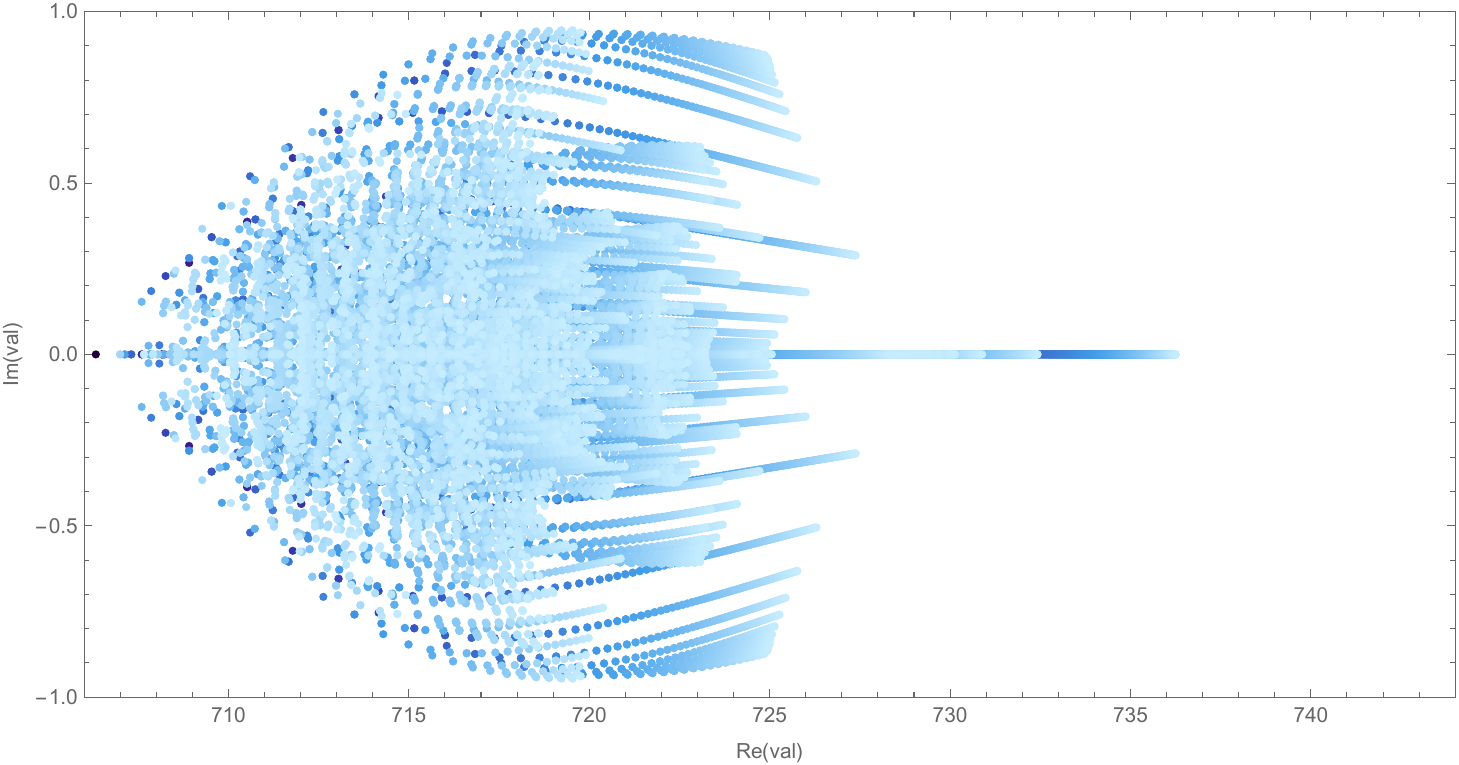}
	\caption{The plot consists of approximately 37,000 points corresponding to conjugacy classes of primitive hyperbolic elements $\gamma \in \Gamma$ satisfying $\ell(C_\gamma) \leq 13$, where the color becomes lighter as the length $\ell(C_\gamma)$ increases.}
	\label{fig:val-distr}
\end{figure}

Here, a hyperbolic element is called \emph{primitive} if it is not a proper power of another hyperbolic element of $\Gamma$. Equivalently, the corresponding closed geodesic is primitive, that is, it is not obtained by traversing a shorter closed geodesic more than once. A similar numerical picture appeared in a proceedings article by Kaneko and Shigeki~\cite{KanekoShigeki2010}, where the distribution problem was also raised. Regarding primitive closed geodesics as analogues of prime numbers, this suggests a distribution problem analogous in spirit to the Sato--Tate problem.

Let $\Pi$ denote the set of conjugacy classes of primitive hyperbolic elements in $\Gamma$, and put
\[
	\Pi(L) \coloneqq \{[\gamma]\in\Pi:\ell(C_\gamma)\leq L\},
	\qquad \#\Pi(L)\sim\mathrm{Li}(e^L)\quad(L\to\infty),
\]
where the asymptotic is the prime geodesic theorem~\cite[(44)]{Sarnak2007}. Our main theorem gives a particularly simple answer: the values of the val function concentrate at $720$. More generally, let $M_0^!(\Gamma)$ denote the space of weakly holomorphic modular functions (of weight $0$) for $\Gamma$. For $f\in M_0^!(\Gamma)$, set
\[
	\val_f(\gamma)\coloneqq\frac{I_f(\gamma)}{I_1(\gamma)},
	\qquad
	E_2(\tau)\coloneqq1-24\sum_{n\geq1}\sigma_1(n)q^n,
\]
where \(\sigma_1(n)\coloneqq\sum_{d\mid n}d\), so that $\val_j=\val$. We write $\alpha_f\coloneqq(f,1)_{\mathrm{At}}$ for Atkin's inner product of $f$ with the constant function $1$. By Kaneko--Zagier~\cite[Proposition~3]{KanekoZagier1998}, $\alpha_f$ is precisely the constant term of $fE_2$. Then the following holds.

\begin{theorem}\label{thm:main1}
	Let $f\in M_0^!(\Gamma)$. For every $\varepsilon > 0$,
	\[
		\lim_{L\to\infty} \frac{\#\{[\gamma] \in \Pi(L) : |\val_f(\gamma)-\alpha_f| < \varepsilon\}}{\#\Pi(L)} =1.
	\]
	In particular, since $\alpha_j=720$, the values of the val function concentrate at $720$.
\end{theorem}

For the different ordering by the discriminant of the associated indefinite binary quadratic form, Duke--Friedlander--Iwaniec~\cite{DFI2012} and Masri~\cite{Masri2012} showed that the average of $\val(\gamma)$ over classes of fundamental discriminant $D$ tends to $720$ as $D\to\infty$. This first-moment result permits cancellation of imaginary parts, whereas our length-ordered theorem describes the distribution of the individual values.

Kaneko and Shigeki~\cite{KanekoShigeki2010} made a further intriguing observation concerning the imaginary parts. To describe it, we recall the following symmetries. For $\gamma = \smat{a&b\\c&d} \in \Gamma$ and $\gamma^\ast \coloneqq \smat{a&-b\\-c&d}$, Kaneko~\cite{Kaneko2009} proved
\begin{align}\label{eq:Kaneko-symmetry}
	\val(\gamma^{-1}) = \val(\gamma), \quad \val(\gamma^\ast) = \overline{\val(\gamma)}.
\end{align}
Hence $\val(\gamma)$ is real whenever $\gamma$ is conjugate in $\Gamma$ to either $\gamma^\ast$ or $(\gamma^\ast)^{-1}$. 

\begin{conjecture}[Kaneko--Shigeki]\label{conj:Kaneko-Shigeki}
	For every hyperbolic element $\gamma\in\Gamma$, the value $\val(\gamma)$ is real if and only if $\gamma$ is conjugate in $\Gamma$ to either $\gamma^\ast$ or $(\gamma^\ast)^{-1}$.
\end{conjecture}

No numerical counterexample to \cref{conj:Kaneko-Shigeki} was found among the primitive classes with $\ell(C_\gamma)\leq13$ used in the computation of \cref{fig:Imval-distr}. Kaneko's original formulation is given in terms of real quadratic irrationalities rather than hyperbolic elements. We verify in~\cref{section:proof-2} that it is equivalent to the formulation above. Numerically, once the classes satisfying this symmetry condition are removed, the sharp spike at $0$ in the distribution of $\Im\val(\gamma)$ disappears, as shown in~\cref{fig:Imval-distr}.

\begin{figure}[H]
	\centering
	\includegraphics[width=10cm]{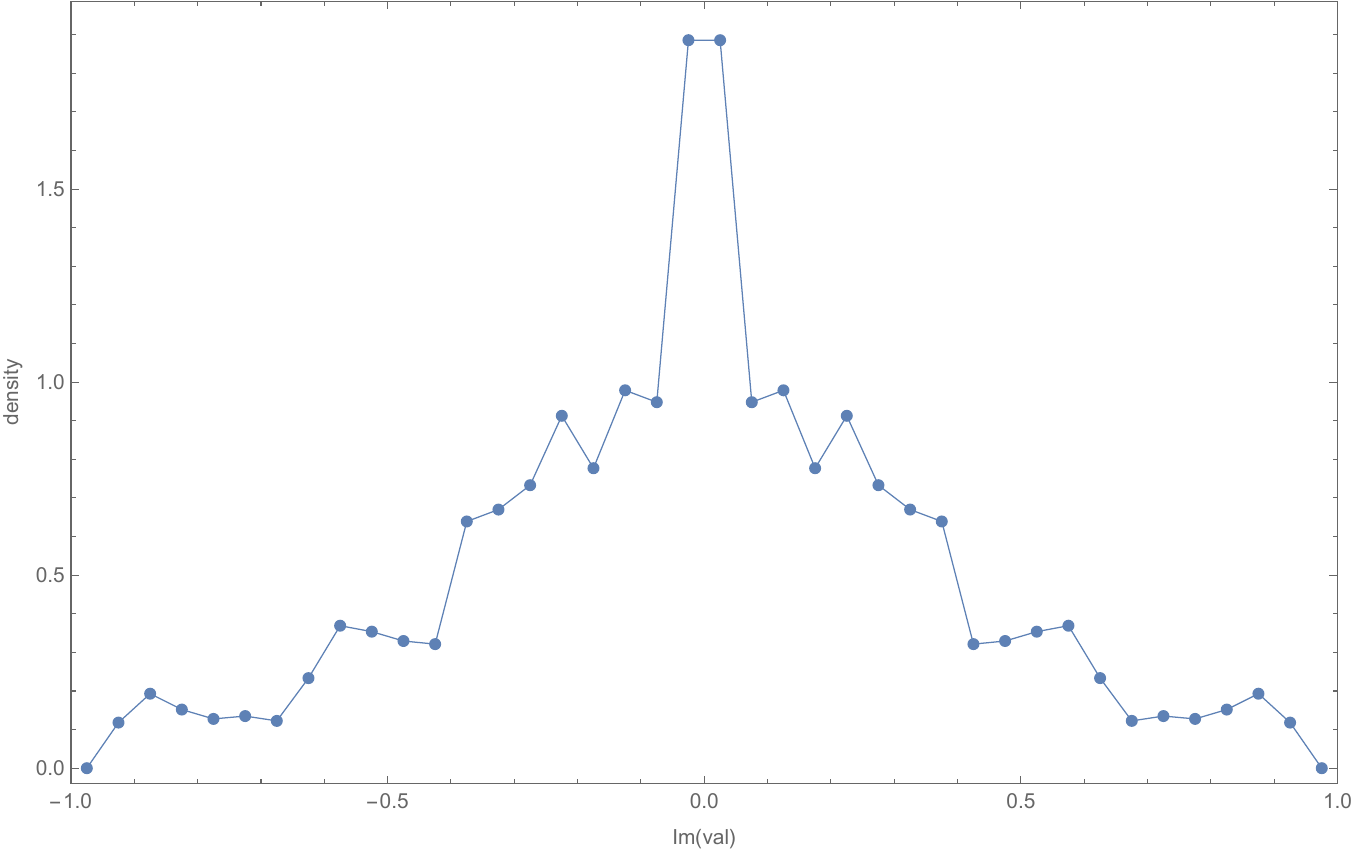}
	\caption{Distribution of $\Im\val(\gamma)$ for conjugacy classes of primitive hyperbolic elements $\gamma\in\Gamma$ satisfying $\ell(C_\gamma)\leq 13$, after excluding approximately $26.5\%$ of the classes, namely those for which $\gamma$ is conjugate in $\Gamma$ to either $\gamma^\ast$ or $(\gamma^\ast)^{-1}$. The interval $(-1,1)$ is divided into $40$ equal subintervals, and the resulting counts are normalized so that the total area is $1$.}
	\label{fig:Imval-distr}
\end{figure}

We do not resolve~\cref{conj:Kaneko-Shigeki}. Nevertheless, assuming it, the concentration at $720$ persists even after all real values are removed. To state this precisely, let
\[
	\Pi_{\mathrm{nr}}(L) \coloneqq \{[\gamma] \in \Pi(L) : \Im\val(\gamma) \neq 0 \}.
\]

\begin{theorem}\label{thm:main2}
	Assume \cref{conj:Kaneko-Shigeki}. Then, for every $\varepsilon>0$,
\[
	\lim_{L \to\infty}
	\frac{\#\left\{[\gamma] \in \Pi_{\mathrm{nr}}(L) : |\val(\gamma)-720|<\varepsilon \right\}}{\#\Pi_{\mathrm{nr}}(L)} = 1.
\]
\end{theorem}

In \cref{section:suspension}, we construct the continued-fraction suspension, prove its ergodicity, and identify its primitive closed orbits with primitive hyperbolic classes, preserving length. A key input for the proof of \cref{thm:main1} is the suspension equidistribution theorem (\cref{thm:pollicott-equidistribution}), which we state in \cref{section:BHI-suspension-observable} in a form following Pollicott's argument. In the same section, using the decomposition formula of Bengoechea--Imamo\={g}lu, we construct a bounded continuous observable on the suspension whose averages along primitive closed orbits recover $\val_f$, and whose space average is $\alpha_f$. These ingredients prove \cref{thm:main1} in \cref{section:proof-1}. Since \cref{thm:pollicott-equidistribution} is not stated in precisely this form in Pollicott's paper, we include an independent proof in \cref{sec:Proof-Pollicott}. Finally, \cref{section:proof-2} compares the Kaneko--Shigeki formulations and proves \cref{thm:main2} using Sarnak's counts.

\section{A suspension model for the modular geodesic flow}\label{section:suspension}

The purpose of this section is to construct a continued-fraction suspension
\[
	(X_{\susp},m_{\susp},(\psi_s)_{s\in\R}),
\]
where $m_{\susp}$ is a probability measure and $(\psi_s)_{s\in\R}$ is an ergodic measure-preserving flow. The construction uses the continued-fraction coordinates of Adler--Flatto~\cite{AdlerFlatto1984} and Series~\cite{Series1985} and models the geodesic flow on $X_{\mathrm{mod}}\coloneqq \Gamma \backslash \PSL_2(\R)$ with normalized Haar measure. Under $\PSL_2(\R)\cong T^1\bbH$, the normalized Haar measure corresponds to the Liouville measure, and the right diagonal flow is the geodesic flow. Its closed orbits, determined by hyperbolic conjugacy classes, are the modular knots of Ghys~\cite{Ghys2007}. We use Einsiedler--Ward~\cite[Chapters~3 and~9]{EinsiedlerWard2011} for the ergodic-theoretic input.

Every topological space below carries its Borel $\sigma$-algebra. An invertible measure-preserving transformation $S\colon(X,m)\to(X,m)$ is \emph{ergodic} if every measurable $A\subset X$ satisfying
\[
	m\bigl(S^{-1}(A)\mathbin{\triangle}A\bigr)=0
\]
has $m(A)=0$ or $1$, where $A\mathbin{\triangle}B\coloneqq(A\setminus B)\cup(B\setminus A)$. A \emph{flow} is a continuous $\R$-action $(\phi_s)_{s\in\R}$, and it is \emph{measure-preserving} if every $\phi_s$ preserves the given Borel probability measure $m$. Following Einsiedler--Ward~\cite[Chapter~8]{EinsiedlerWard2011}, such a flow is \emph{ergodic} if every measurable $A\subset X$ satisfying
\[
	m\bigl(\phi_s^{-1}(A) \mathbin{\triangle} A\bigr) = 0 \quad \text{for every } s \in \R
\]
has $m(A)=0$ or $1$.

\subsection{Natural extension of the Gauss map}

Let $I\coloneqq(0,1)\setminus\Q$. Every $x\in I$ has a unique continued-fraction expansion
\[
	x= [a_1,a_2,\ldots] \coloneqq \cfrac{1}{a_1+\cfrac{1}{a_2+\ddots}}.
\]
Define the Gauss map $T\colon I\to I$ by
\[
	T(x)\coloneqq\frac1x-\left\lfloor\frac1x\right\rfloor =[a_2,a_3,\ldots].
\]
Put $\widehat{I} \coloneqq I\times I$. For $(x,y)=\bigl([a_1,a_2,\ldots],[a_0,a_{-1},\ldots]\bigr)\in\widehat{I}$, define its \emph{natural extension} $\widehat{T}\colon\widehat{I} \to \widehat{I}$ by
\[
	\widehat{T}(x,y)
	=\bigl([a_2,a_3,\ldots],[a_1,a_0,a_{-1},\ldots]\bigr)
	=\left(Tx,\frac1{\lfloor1/x\rfloor+y}\right).
\]

\begin{lemma}\label{lem:natural-extension}
	The map $\widehat{T}$ is a measure-preserving homeomorphism of $\widehat{I}$ with respect to the probability measure $\widehat{\mu}$ defined by
	\[
		\widehat{\mu}(A) \coloneqq \frac{1}{\log 2} \int_A \frac{1}{(1+xy)^2}\,\dd x\,\dd y
	\]
	for every Borel set $A \subset \widehat{I}$.
\end{lemma}

\begin{proof}
	For $a\ge1$, put
	\begin{align}\label{eq:def-Ia}
		I_a\coloneqq \left(\frac{1}{a+1},\frac{1}{a}\right)\cap I.
	\end{align}
	The open sets $I_a$ partition $I$, and $\widehat{T}$ maps $I_a\times I$ bijectively onto $I\times I_a$. On this branch, writing $(x',y')=\widehat{T}(x,y)$, we have
	\[
		(x',y')=\left(\frac1x-a,\frac1{a+y}\right),
		\qquad
		\widehat{T}^{-1}(x',y')=\left(\frac1{a+x'},\frac1{y'}-a\right),
		\qquad
		\frac{\dd x'\,\dd y'}{(1+x'y')^2}=\frac{\dd x\,\dd y}{(1+xy)^2}.
	\]
	The first two formulas show that $\widehat{T}$ is a homeomorphism, and the last one and change of variables give
	\begin{align*}
		\widehat{\mu}(\widehat{T}^{-1}A) &= \frac{1}{\log 2} \sum_{a\ge1} \int_{\widehat{T}^{-1}A\cap(I_a\times I)} \frac{1}{(1+xy)^2}\,\dd x\,\dd y \\
			&= \frac{1}{\log 2} \sum_{a\ge1} \int_{A\cap(I\times I_a)} \frac{1}{(1+x'y')^2}\,\dd x'\,\dd y' = \widehat{\mu}(A).
	\end{align*}
	Finally,
	\[
		\widehat{\mu}(\widehat{I}) = \frac{1}{\log 2} \int_0^1\int_0^1 \frac{1}{(1+xy)^2}\,\dd x\,\dd y = \frac{1}{\log 2} \int_0^1\frac{1}{1+x}\,\dd x =1
	\]
	proves the normalization.
\end{proof}

\subsection{Parity extension}

Adjoin a parity coordinate and equip $\Z/2\Z=\{0,1\}$ with the discrete topology
and the normalized counting measure $m_{\Z/2\Z}$. Set
\[
	\Sigma\coloneqq\widehat{I}\times\Z/2\Z, \qquad
	m_\Sigma\coloneqq\widehat{\mu}\otimes m_{\Z/2\Z}, \qquad
	\sigma(x,y,\epsilon)\coloneqq \bigl(\widehat{T}(x,y),\epsilon+1\bigr).
\]
By \cref{lem:natural-extension}, $\sigma$ is a measure-preserving homeomorphism of $\Sigma$.

\begin{lemma}\label{lem:sigma-ergodic}
	The transformation $\sigma\colon(\Sigma,m_\Sigma)\to(\Sigma,m_\Sigma)$ is ergodic.
\end{lemma}

\begin{proof}
	Let
	\[
		\overline{Y} \coloneqq \left\{(x,z)\in[0,1)^2: 0\leq z \leq \frac{1}{1+x}\right\},
	\]
	and equip $\overline{Y}$ with the probability measure $\dd m_{\overline{Y}} \coloneqq \dd x\,\dd z/\log 2$. Einsiedler--Ward~\cite[Proposition~3.15]{EinsiedlerWard2011} consider another invertible extension $\overline{T} \colon \overline{Y} \to \overline{Y}$ of the Gauss map given, modulo null boundary sets, by $\overline{T}(x,z) \coloneqq (Tx,x(1-xz))$. Einsiedler--Ward use the parity set $\{\pm1\}$, and under the bijection $\epsilon\mapsto(-1)^\epsilon$ their sign change corresponds to $\epsilon\mapsto\epsilon+1$ in $\Z/2\Z$. Thus their parity extension is the transformation $\widetilde{T} \colon \overline{Y} \times \Z/2\Z \to \overline{Y} \times \Z/2\Z$ given by
	\[
		\widetilde{T}(x,z,\epsilon) \coloneqq \bigl(\overline{T}(x,z),\epsilon+1\bigr).
	\]
	The proof of Einsiedler--Ward~\cite[Proposition~9.25]{EinsiedlerWard2011} shows that $\widetilde T$ is ergodic with respect to $m_{\overline{Y}}\otimes m_{\Z/2\Z}$.

	Consider the coordinate change $\kappa \colon \Sigma \to \overline{Y} \times \Z/2\Z$ defined by
	\[
		\kappa(x,y,\epsilon) \coloneqq \left(x,\frac{y}{1+xy},\epsilon\right).
	\]
	Put
	\[
		Y_0 \coloneqq \kappa(\Sigma) = \left\{(x,z,\epsilon)\in\overline{Y}\times\Z/2\Z: x\in I,\ 0<z<\frac{1}{1+x},\ \frac{z}{1-xz}\in I\right\}.
	\]
	This description shows that $Y_0$ is Borel. Since
	\[
		\kappa^{-1}(x,z,\epsilon) = \left(x,\frac{z}{1-xz},\epsilon\right),
	\]
	the map $\kappa \colon \Sigma \to Y_0$ is a bimeasurable bijection. A direct calculation shows that $\kappa_*m_\Sigma = m_{\overline{Y}} \otimes m_{\Z/2\Z}$ and $\kappa \circ \sigma = \widetilde{T} \circ \kappa$ on $\Sigma$. The first identity shows that $Y_0$ is conull, and the second that it is invariant. Hence the two systems are isomorphic in the sense of Einsiedler--Ward~\cite[Definition~2.7(2)]{EinsiedlerWard2011}. Since $\widetilde T$ is ergodic, so is $\sigma$.
\end{proof}

\subsection{Suspension flow}

We now construct the suspension flow over $(\Sigma,m_\Sigma,\sigma)$. For $\omega = (x,y,\epsilon) \in \Sigma$, write $(x',y') = \widehat{T}(x,y)$, and define the \emph{roof function} $r \colon \Sigma \to \R_{>0}$ by
\[
	r(x,y,\epsilon) \coloneqq -\frac{1}{2} \log(xyx'y').
\]
We note that $r$ is independent of the parity coordinate $\epsilon$.

\begin{lemma}\label{lem:roof-function}
	The roof function $r$ is continuous and satisfies
	\[
		r(\omega) > \log 2
	\]
	for every $\omega \in \Sigma$. Moreover, $r\in L^1(\Sigma, m_\Sigma)$ and
	\[
		\overline{r} \coloneqq \int_\Sigma r\,\dd m_\Sigma = \frac{\pi^2}{6\log 2}.
	\]
	If $\omega = (x,y,\epsilon) \in \Sigma$ and $a = \lfloor 1/x \rfloor \geq 2$, then $r(\omega)\geq\log a$.
\end{lemma}

\begin{proof}
	By definition, if $x=1/(a+x')$, then $y'=1/(a+y)$ for some $a\ge1$. Hence 
	\[ 
		xyx'y' = \frac{x'}{a+x'}\frac{y}{a+y},
	\] 
	and therefore
	\[
		r(x,y,\epsilon) = \frac{1}{2} \log \left(\frac{a+x'}{x'} \frac{a+y}{y}\right).
	\]
	Since $a \geq 1$ and $0< x',y <1$, each of the two factors inside the logarithm is greater than $2$, and hence $r(x,y,\epsilon) > \log 2$. For each $a\geq1$, the set $I_a$ defined in \eqref{eq:def-Ia} is open in $I$, and on $I_a \times I \times \Z/2\Z$ the integer $a$ is constant and $x'=1/x-a$ is continuous. The formula above therefore shows that $r$ is continuous on each of these open sets, and hence on all of $\Sigma$. Moreover, $0<x',y<1$ makes each factor greater than $a$. Hence $r(x,y,\epsilon)>\log a$, which gives the stated inequality when $a\geq2$.
	
	It remains to prove integrability and compute the mean of $r$. Define 
	\[
		H(x,y) \coloneqq -\frac{1}{2}\log x +\frac{1}{2}\log y -\log(1+xy).
	\]
	If $(x',y')=\widehat{T}(x,y)$, then
	\[
		\frac{y'}{x} = \frac{1+x'y'}{1+xy},
	\]
	and hence
	\[
		\log(1+xy)-\log(1+x'y') = \log x-\log y'.
	\]
	It follows that
	\begin{align}\label{eq:roof-coboundary} 
		r(x,y,\epsilon) = -2\log x + H(x',y') - H(x,y). 
	\end{align} 
	We have 
	\[ 
		\int_{\widehat{I}}|\log x|\,\dd\widehat{\mu} = \frac{1}{\log 2} \int_0^1\frac{-\log x}{1+x}\,\dd x <\infty,
	\]
	and, by symmetry, $\log x, \log y \in L^1(\widehat{I}, \widehat{\mu})$. Since $\log(1+xy)$ is bounded on $\widehat{I}$, it follows that $H \in L^1(\widehat{I}, \widehat{\mu})$. Hence \eqref{eq:roof-coboundary} and the $\widehat{T}$-invariance of $\widehat{\mu}$ imply that $r \in L^1(\Sigma,m_\Sigma)$ and 
	\begin{align*} 
		\overline{r} &= -2\int_{\widehat{I}}\log x\,\dd\widehat{\mu} +\int_{\widehat{I}} \bigl(H\circ\widehat{T} - H\bigr)\,\dd\widehat{\mu}\\
			&= -\frac{2}{\log 2} \int_0^1\frac{\log x}{1+x}\,\dd x = \frac{\pi^2}{6\log 2}.
	\end{align*}
	Here the last equality follows by termwise integration of the geometric series for $(1+x)^{-1}$.
\end{proof}

\begin{definition}\label{def:m-susp}
	We define the \emph{suspension space} $X_{\susp}$ over $\sigma$ with roof function $r$ by
	\[
		X_{\susp} \coloneqq (\Sigma\times\R) \big/\sim,
	\]
	where the equivalence relation $\sim$ is generated by $(\omega, t + r(\omega)) \sim (\sigma \omega, t)$ for $\omega \in \Sigma$ and $t \in \R$. We equip $X_{\susp}$ with the quotient topology and write $[\omega,t]$ for the equivalence class of $(\omega,t)$. We define a probability measure $m_{\susp}$ on $X_{\susp}$ by
	\[
		m_{\susp}(A) \coloneqq \frac{1}{\overline{r}} \int_\Sigma \int_0^{r(\omega)} \mathbf{1}_A([\omega,t])\,\dd t\,\dd m_\Sigma(\omega)
	\]
	for every Borel set $A\subset X_{\susp}$. For $s \in \R$, define
	\[
		\psi_s([\omega,t]) \coloneqq [\omega,t+s].
	\]
\end{definition}

We next verify that $(\psi_s)_{s\in\R}$ is a measure-preserving flow and then establish its ergodicity.

\begin{lemma}\label{lem:suspension-flow}
	The family $(\psi_s)_{s\in\R}$ is a measure-preserving flow on $(X_{\susp},m_{\susp})$.
\end{lemma}

\begin{proof}
	We first verify that $(\psi_s)_{s\in\R}$ is a flow. Let $\pi \colon \Sigma \times \R \to X_{\susp}$ be the quotient map, and define $\widetilde{\sigma} \colon \Sigma \times \R \to \Sigma \times \R$ by
	\[
		\widetilde{\sigma}(\omega, t) \coloneqq (\sigma\omega, t-r(\omega)).
	\]
	Since $\sigma$ is a homeomorphism and $r$ is continuous by \cref{lem:roof-function}, the map $\widetilde{\sigma}$ is a homeomorphism. The equivalence classes defining $X_{\susp}$ are precisely the $\Z$-orbits of $\widetilde{\sigma}$. Hence, for every open set $U\subset\Sigma\times\R$,
	\[
		\pi^{-1}(\pi(U)) = \bigcup_{n\in\Z} \widetilde{\sigma}^n(U),
	\]
	which is open. By the definition of the quotient topology, $\pi(U)$ is open. Thus $\pi$ is an open quotient map.

	Consequently, $\mathrm{id}_{\R} \times \pi \colon \R \times \Sigma \times \R \to \R \times X_{\susp}$ is also an open quotient map. Define $\widetilde{\Psi} \colon \R \times \Sigma \times \R \to X_{\susp}$ by $(s,\omega,t)\mapsto[\omega,t+s]$. This map is continuous, being the composition of $\pi$ with the continuous map $(s, \omega, t) \mapsto (\omega, t+s)$. Since the equivalence relation defining $X_{\susp}$ is invariant under translation in the second coordinate, $\widetilde{\Psi}$ is constant on the fibers of $\id_{\R}\times\pi$. Hence, by the universal property of the quotient topology, it induces a continuous map $\Psi \colon \R \times X_{\susp} \to X_{\susp}$ satisfying $\Psi(s,[\omega,t]) = [\omega,t+s]$. Thus $\psi_s$ is well-defined and $(s,z) \mapsto \psi_s(z)$ is continuous. Moreover,
	\[
		\psi_0=\mathrm{id}_{X_{\susp}}, \qquad \psi_{s+t}=\psi_s\circ\psi_t
	\]
	for all $s,t\in\R$. Thus $(\psi_s)_{s\in\R}$ is a flow.
	
	It remains to prove that the flow is measure-preserving. We apply Einsiedler--Ward~\cite[Lemma~9.23]{EinsiedlerWard2011} to the special flow over $\sigma$ with roof function $r$. The required hypotheses are satisfied: the shift map $\sigma \colon (\Sigma, m_\Sigma) \to (\Sigma, m_\Sigma)$ is invertible and measure-preserving, while by \cref{lem:roof-function}, the roof function $r$ is positive, Borel measurable, and integrable. Put
	\[
		X_r\coloneqq\{(\omega,t)\in\Sigma\times\R:0\leq t<r(\omega)\}
	\]
	and define $\Theta\colon X_r\to X_{\susp}$ by $\Theta(\omega,t)\coloneqq[\omega,t]$. Every equivalence class has a unique representative in $X_r$, so $\Theta$ is a bijection. By
Einsiedler--Ward~\cite[Lemma~9.23]{EinsiedlerWard2011}, the finite suspension measure on $X_r$ is invariant under the special flow. By the definition of $m_{\susp}$, this measure pushes forward under $\Theta$ to $\overline r\,m_{\susp}$, and $\Theta$ intertwines the special flow on $X_r$ with $(\psi_s)_{s\in\R}$. Hence $\overline r\,m_{\susp}$, and therefore its normalization $m_{\susp}$, is invariant under $(\psi_s)_{s\in\R}$. Thus the flow is measure-preserving.
\end{proof}

We call $(\psi_s)_{s\in\R}$ the \emph{suspension flow} over $\sigma$ with roof function $r$.

\begin{lemma}\label{lem:suspension-ergodic}
	The suspension flow $(\psi_s)_{s\in\R}$ on $(X_{\susp},m_{\susp})$ is ergodic.
\end{lemma}

\begin{proof}
	By \cref{lem:sigma-ergodic}, the transformation $\sigma\colon(\Sigma,m_\Sigma)\to(\Sigma,m_\Sigma)$ is ergodic. Moreover, $r$ is integrable by \cref{lem:roof-function}. Hence Einsiedler--Ward~\cite[Lemma~9.24]{EinsiedlerWard2011} implies that the corresponding special flow is ergodic with respect to its finite suspension measure. Since $m_{\susp}$ is the normalization of this measure, the suspension flow is ergodic with respect to $m_{\susp}$.
\end{proof}

\begin{remark}
	The ergodicity of $(\psi_s)_{s\in\R}$ ultimately comes from Hopf's ergodicity theorem for the modular geodesic flow with respect to Haar measure, see Einsiedler--Ward~\cite[Theorem~9.21 and the proof of Proposition~9.25]{EinsiedlerWard2011} for the connection with the corresponding continued-fraction suspension.
\end{remark}

\subsection{Closed orbits}

By a closed orbit of $(\psi_s)_{s\in\R}$ we mean an orbit
\[
	C=\{\psi_s z:s\in\R\}\subset X_{\susp}
\]
for which $\psi_Tz=z$ for some $T>0$. Its length $\ell(C)$ is the least positive such $T$.

Series~\cite{Series1985} gives the classical correspondence used here. If $\sigma^N\omega=\omega$ and $\sigma^j\omega=(x_j,y_j,\epsilon_j)$, then $N$ is even because $\sigma$ flips parity, and \eqref{eq:roof-coboundary} telescopes to
\begin{equation}\label{eq:periodic-roof-sum}
	\sum_{j=0}^{N-1}r(\sigma^j\omega)=-2\sum_{j=0}^{N-1}\log x_j.
\end{equation}
If $[\omega,0]\in C$ and $N$ is the least positive integer with $\sigma^N\omega=\omega$, then
\[
	\ell(C)=\sum_{j=0}^{N-1}r(\sigma^j\omega).
\]
A closed orbit is \emph{primitive} if it is not an iterate of a shorter one. Let $\calP$ be the set of primitive closed orbits.

In $\Gamma=\PSL_2(\Z)$, put
\[
	\matS\coloneqq\pmat{0&-1\\1&0},\qquad
	\matU\coloneqq\pmat{0&-1\\1&1},\qquad
	\matT\coloneqq\matS\matU=\pmat{1&1\\0&1},\qquad
	\matV\coloneqq\matS\matU^2=\pmat{1&0\\1&1}.
\]
The presentation
\[
	\Gamma=\langle\matS,\matU\mid\matS^2=\matU^3=1\rangle
	\simeq\Z/2\Z*\Z/3\Z
\]
is given in Serre~\cite[Chapter VII]{Serre1973}. 

Since $\matT=\matS\matU$ and $\matV=\matS\matU^2$, every cyclically reduced word of length greater than one is, up to cyclic shift, a cyclic $\matT/\matV$-word. The unmixed words are parabolic, whereas every mixed word is hyperbolic. The standard normal-form theory for free products then shows that two mixed cyclic $\matT/\matV$-words are conjugate exactly when they agree cyclically, and that the represented element is a proper power exactly when the cyclic word is a proper repetition. Decomposition into maximal cyclic runs
\[
	R_1^{a_1}\cdots R_N^{a_N},
	\qquad R_j \in \{\matT,\matV\},\quad a_j\geq1,
\]
is unique up to cyclic shift. By maximality, $R_j\ne R_{j+1}$ cyclically, including $R_N \ne R_1$. Thus the runs alternate between $\matT$ and $\matV$, so $N$ is even.

\begin{proposition}\label{prop:closed-orbit-dictionary}
	There is a natural one-to-one correspondence
	\[
		\mathcal P \longleftrightarrow \Pi,
	\]
	where $\Pi$ is the set of conjugacy classes of primitive hyperbolic elements in $\Gamma$.
\end{proposition}

\begin{proof}
	For $C \in \calP$, we first construct $[\gamma_C] \in \Pi$. Choose $[\omega,0]\in C$, and let $N$ be the least positive period of $\omega$ under $\sigma$. If $a_j$ is the digit encountered at the $j$th step, then
	\[
		\omega=\left([\overline{a_1,\ldots,a_N}],[\overline{a_N,\ldots,a_1}],\epsilon\right),
	\]
	and define the mixed cyclic word
	\[
		\gamma_\omega \coloneqq \begin{cases}
			\matV^{a_1}\matT^{a_2} \cdots \matV^{a_{N-1}}\matT^{a_N}, &\text{if } \epsilon=0,\\[2mm]
			\matT^{a_1}\matV^{a_2} \cdots \matT^{a_{N-1}}\matV^{a_N}, &\text{if } \epsilon=1.
		\end{cases}
	\]
	The displayed factors $\matV^{a_1},\matT^{a_2},\ldots$ or $\matT^{a_1},\matV^{a_2},\ldots$ are precisely the maximal cyclic runs of the associated cyclic word. Replacing $\omega$ by $\sigma^j\omega$ cyclically shifts these runs, so the conjugacy class $[\gamma_C]\coloneqq[\gamma_\omega]$ depends only on $C$.

	If $\gamma_\omega$ were a proper power, its cyclic run word would have a proper period $d<N$. Since the runs alternate between $\matT$ and $\matV$, $d$ is even, and hence $\sigma^d\omega=\omega$, contradicting the minimality of $N$. Thus $\gamma_\omega$ is primitive and $[\gamma_C]\in\Pi$.

	We next show that this correspondence is bijective. Conversely, let $[\gamma]\in\Pi$ and write a mixed cyclic representative as $R_1^{a_1}\cdots R_N^{a_N}$. Set
	\[
		\epsilon\coloneqq \begin{cases}
			0, &\text{if } R_1=\matV,\\
			1, &\text{if } R_1=\matT,
		\end{cases}
		\qquad
		\omega\coloneqq \left([\overline{a_1,\ldots,a_N}], [\overline{a_N,\ldots,a_1}], \epsilon \right).
	\]
	Then $\sigma^N \omega = \omega$ and $\gamma_\omega$ represents $[\gamma]$. If $\omega$ had a smaller $\sigma$-period $d<N$, then $d$ would be even and the cyclic run word of $\gamma$ would have period $d$, contradicting the primitivity of $\gamma$. Hence the corresponding orbit lies in $\calP$ and maps to $[\gamma]$, proving surjectivity.

	Finally, choose $[\omega,0]\in C$ and $[\omega',0]\in C'$ as above. If $[\gamma_C]=[\gamma_{C'}]$, uniqueness of the maximal cyclic run decomposition implies that the corresponding digit and parity data differ by a cyclic shift. Hence $\omega'=\sigma^j\omega$ for some $j$, so $C=C'$. This proves injectivity.
\end{proof}

For a hyperbolic element $\gamma \in \Gamma$, we write $\ell(\gamma) \coloneqq \ell(C_\gamma)$ for the length of the associated closed geodesic on the modular surface. Since this depends only on the conjugacy class of $\gamma$, the quantity $\ell(\gamma_C)$ is well-defined for $C\in\mathcal P$.

\begin{lemma}\label{lem:length-preserving}
	The correspondence in \cref{prop:closed-orbit-dictionary} preserves lengths. More precisely, for every $C\in\mathcal P$, we have
	\[
		\ell(C)=\ell(\gamma_C).
	\]
\end{lemma}

\begin{proof}
	Choose $[\omega,0]\in C$ with parity coordinate $1$, let $N$ be its least $\sigma$-period, and write
	\[
		\omega=\left([\overline{a_1,\ldots,a_N}],[\overline{a_N,\ldots,a_1}],1\right),
		\qquad \sigma^j\omega=(x_j,y_j,\epsilon_j).
	\]
	By \eqref{eq:periodic-roof-sum},
	\[
		\ell(C)=\sum_{j=0}^{N-1}r(\sigma^j\omega)
		=-2\sum_{j=0}^{N-1}\log x_j.
	\]
	Bengoechea--Herrero--Imamo\={g}lu~\cite[Proposition~3.2(i)]{BHI2026pre} give the following formula. In the notation of that proposition, $v_j=x_j^{-1}$:
	\[
		\ell(\gamma_C)=2\sum_{j=0}^{N-1}\log x_j^{-1},
	\]
	which is the same quantity.
\end{proof}

For $L>0$, put
\[
	\calP(L) \coloneqq \{C\in\calP : \ell(C) \leq L\}.
\]
By \cref{prop:closed-orbit-dictionary,lem:length-preserving}, $C\mapsto[\gamma_C]$ restricts to a bijection $\calP(L) \to \Pi(L)$. In particular, $\#\calP(L) = \#\Pi(L)$.

\section{Bengoechea--Imamo\={g}lu's suspension observable}\label{section:BHI-suspension-observable}

Let $C_b(X_{\susp})$ be the space of bounded continuous functions $X_{\susp}\to\C$. For $C\in\calP$ and $F \in C_b(X_{\susp})$, put
\[
	m_C(F) \coloneqq \frac{1}{\ell(C)} \int_0^{\ell(C)} F(\psi_s z)\,\dd s
\]
for any $z\in C$. This is independent of $z$.

\begin{theorem}[Suspension equidistribution, following Pollicott~\cite{Pollicott1986}]\label{thm:pollicott-equidistribution}
	For every $F \in C_b(X_{\susp})$,
	\[
		\lim_{L\to\infty} \frac{1}{\#\calP(L)} \sum_{C \in \calP(L)} m_C(F) = \int_{X_{\susp}} F\,\dd m_{\susp}.
	\]
\end{theorem}

\begin{remark}\label{rem:pollicott-route}	
	Pollicott~\cite[Theorem 1 and its Remark]{Pollicott1986} states the corresponding equidistribution theorem for closed geodesics on the modular surface rather than literally in the form above. His proof is carried out through the associated continued-fraction suspension and Mayer's transfer operators. The same argument gives the formulation for bounded continuous observables on $X_{\susp}$ used here. In \cref{sec:Proof-Pollicott}, we give an alternative proof based on Ustinov's effective counting theorem for reduced quadratic irrationals.
\end{remark}

Using Bengoechea--Imamo\={g}lu's decomposition~\cite{BI2020}, we construct $\calJ_f\in C_b(X_{\susp})$ for $f\in M_0^!(\Gamma)$ such that
\[
	m_C(\calJ_f)=\val_f(\gamma_C),\qquad
	\int_{X_{\susp}}\calJ_f\,\dd m_{\susp}=\alpha_f.
\]

\subsection{Bengoechea--Imamo\={g}lu's decomposition formula}

Bengoechea--Imamo\={g}lu's decomposition originates in~\cite[Lemma~4.1]{BI2019}. We use the regular-continued-fraction form in~\cite[Remark~2.4 and Lemma~2.5]{BHI2026}, translated to our convention of \cref{section:suspension}. Let
\[
	\calA \coloneqq \left\{e^{it} : \frac{\pi}{3}\leq t\leq\frac{2\pi}{3}\right\},
\]
oriented from $e^{\pi i/3}$ to $e^{2\pi i/3}$. For $\gamma=\matV^{a_1}\matT^{a_2}\cdots\matV^{a_{N-1}}\matT^{a_N}$, where $N$ is even and every $a_i\geq1$, put $s(\gamma)=\sum_i a_i$ and expand $\gamma=g_1\cdots g_{s(\gamma)}$ into $\matT$- and $\matV$-letters. For $1\leq k\leq s(\gamma)$, let
\[
	\gamma^{(k)} \coloneqq g_k g_{k+1} \cdots g_{s(\gamma)} g_1\cdots g_{k-1},
\]
and denote its attracting and repelling fixed points by $w^{(k)}$ and $\widetilde w^{(k)}$, respectively. A cyclic change of the initial letter only permutes these pairs.

\begin{proposition}[Bengoechea--Imamo\={g}lu]\label{prop:BI-decomposition}
	For every $f\in M_0^!(\Gamma)$,
	\begin{align}\label{eq:BHI-global}
		I_f(\gamma) &= \int_{\calA} f(z) \sum_{k=1}^{s(\gamma)} \left( \frac{1}{z-w^{(k)}}-\frac{1}{z-\widetilde w^{(k)}} \right)\,\dd z.
	\end{align}
\end{proposition}

We group the letter-by-letter terms in \eqref{eq:BHI-global} into the $N$ successive $\matT$- or $\matV$-runs recorded by $\sigma$.

\begin{definition}\label{def:BHI-block}
	For $\omega = (x,y,\epsilon) \in \Sigma$, put $a(\omega) \coloneqq \lfloor 1/x \rfloor$ and $u(\omega) \coloneqq Tx$. For $a \geq 1$, $0 < u,y < 1$, and $z \in \calA$, define
	\begin{align*}
		B_a^{(0)}(z;u,y) &\coloneqq \sum_{m=1}^a \left( \frac{1}{z-(m+u)^{-1}}-\frac{1}{z+(a-m+y)^{-1}} \right),\\
		B_a^{(1)}(z;u,y) &\coloneqq \sum_{m=1}^a \left( \frac{1}{z-(m+u)}-\frac{1}{z+(a-m+y)} \right).
	\end{align*}
	For $f\in M_0^!(\Gamma)$, define
	\[
		\Phi_f(\omega) \coloneqq \int_{\calA} f(z) B_{a(\omega)}^{(\epsilon)} \bigl(z;u(\omega),y\bigr)\,\dd z.
	\]
\end{definition}

For example, take $\gamma=\matV^2\matT=\smat{1&1\\2&3}$. Its three cyclic fixed points are
\[
	w^{(1)}=[\overline{2,1}],\qquad w^{(2)}=[1,\overline{1,2}],\qquad w^{(3)}=[0,\overline{1,2}],
\]
corresponding to $\matV^2\matT$, $\matV\matT\matV$, and $\matT\matV^2$. The associated $\sigma$-orbit and its two blocks are
\[
\begin{gathered}
	\omega=(x,y,0)=([\overline{2,1}],[\overline{1,2}],0),\qquad
	\sigma\omega=(x',y',1)=([\overline{1,2}],[\overline{2,1}],1),\\
	\sum_{k=1}^2\left(\frac1{z-w^{(k)}}-\frac1{z-\widetilde w^{(k)}}\right)
	=B_{a(\omega)}^{(0)}(z;u(\omega),y),\\
	\frac1{z-w^{(3)}}-\frac1{z-\widetilde w^{(3)}}
	=B_{a(\sigma\omega)}^{(1)}(z;u(\sigma\omega),y').
\end{gathered}
\]
Thus the three letter terms become the two run contributions.

More generally, the preceding decomposition takes the following form. Given $C \in \calP$, choose $[\omega,0]\in C$ and let $N$ be the least positive $\sigma$-period of $\omega$. Put
\[
	O_C \coloneqq \{\omega, \sigma\omega, \ldots, \sigma^{N-1}\omega\} \subset \Sigma.
\]
This set is independent of the chosen point on $C$.

\begin{proposition}[Reformulation of \cref{prop:BI-decomposition}]
\label{prop:BHI-suspension}
	For every $f\in M_0^!(\Gamma)$ and $C\in\calP$,
	\[
		I_f(\gamma_C) = \sum_{\omega\in O_C}\Phi_f(\omega).
	\]
\end{proposition}

\begin{proof}
	Choose $\omega_0\in O_C$ with parity $0$ and write
	\[
		\omega_i\coloneqq\sigma^i\omega_0 =(x_i,y_i,\epsilon_i), \qquad (0 \leq i < N).
	\]
	Put $a_i=a(\omega_i)$ and $u_i=u(\omega_i)$. Translating~\cite[Remark~2.4]{BHI2026} to our convention, the $a_i$ pairs in the run represented by $\omega_i$ are, up to order,
	\[
		\begin{cases}
			\bigl((m+u_i)^{-1},-(a_i-m+y_i)^{-1}\bigr), &\text{if } \epsilon_i=0,\\[2mm]
			\bigl(m+u_i,-(a_i-m+y_i)\bigr), &\text{if } \epsilon_i=1,
		\end{cases}
		\qquad (1\leq m\leq a_i).
	\]
	Thus the run contributes $B_{a_i}^{(\epsilon_i)}(z;u_i,y_i)$. These blocks partition the cyclic word and exhaust exactly the $s(\gamma_C)=\sum_i a_i$ terms of \eqref{eq:BHI-global}. Substituting them there gives
	\begin{align*}
		I_f(\gamma_C) = \sum_{i=0}^{N-1} \int_{\calA} f(z)B_{a_i}^{(\epsilon_i)}(z;u_i,y_i)\,\dd z = \sum_{i=0}^{N-1}\Phi_f(\omega_i) = \sum_{\omega\in O_C}\Phi_f(\omega)
	\end{align*}
	as required.
\end{proof}

\subsection{Bounded observable}

We first bound the block observable.

\begin{lemma}\label{lem:Phi-bound}
	For every $f\in M_0^!(\Gamma)$, there exists a constant $C_f > 0$ such that
	\[
		|\Phi_f(\omega)|\leq C_f r(\omega)
	\]
	for every $\omega\in\Sigma$. Moreover, $\Phi_f/r \in C_b(\Sigma)$.
\end{lemma}

\begin{proof}
	Fix $\omega=(x,y,\epsilon)$, put $a=a(\omega)$ and $u=u(\omega)$, and note that $|\Re z|\leq1/2$, $\Im z\geq\sqrt3/2$, and $|z|=1$ on $\calA$. Uniformly in $u,y\in(0,1)$,
	\begin{align*}
		\left|\frac1{z-(m+u)}\right|&\leq\frac1{m-1/2},\\
		\left|\frac1{z+(a-m+y)}\right|&\leq
		\begin{cases}2/\sqrt3,&m=a,\\(a-m-1/2)^{-1},&m<a,
		\end{cases}\\
		\frac1{z-(m+u)^{-1}}-\frac1{z+(a-m+y)^{-1}}
		&=\left(\frac1{z-(m+u)^{-1}}-\frac1z\right) +\left(\frac1z-\frac1{z+(a-m+y)^{-1}}\right).
	\end{align*}
	The two terms on the last line are respectively $O(m^{-1})$ and $O((a-m)^{-1})$ when $m<a$. At $m=a$ the latter has absolute value $|1+yz|^{-1}\leq2$. Summing gives, for both $\epsilon=0,1$,
	\begin{align}\label{eq:uniform-estimate}
		\sup_{z\in\calA}|B_a^{(\epsilon)}(z;u,y)|\ll1+\log a.
	\end{align}
	Since $\sup_{z\in \calA}|f(z)|<\infty$, it follows that
	\[
		|\Phi_f(\omega)|\ll_f 1+\log a(\omega).
	\]
	By \cref{lem:roof-function}, $r>\log 2$ and $r(\omega)\geq\log a(\omega)$ for $a(\omega)\geq2$, proving the asserted bound. On each $I_a\times I\times\Z/2\Z$, both $a(\omega)$ and the finite sum are fixed in form and $u(\omega)=Tx$ is continuous. Hence $\Phi_f$ is continuous. Positivity and continuity of $r$ now give $\Phi_f/r\in C_b(\Sigma)$.
\end{proof}

Fix $\rho(t)=6t(1-t)$, so that $\rho(0)=\rho(1)=0$ and $\int_0^1\rho(t)\,\dd t=1$.

\begin{definition}\label{def:Jf}
	For $f\in M_0^!(\Gamma)$, define $\calJ_f \colon X_{\susp} \to \C$ by
	\[
		\calJ_f([\omega,t]) \coloneqq \frac{\Phi_f(\omega)}{r(\omega)} \rho\left(\frac{t}{r(\omega)}\right)
	\]
	for $0\leq t\leq r(\omega)$.
\end{definition}

\begin{proposition}\label{prop:Jf}
	The function $\calJ_f$ belongs to $C_b(X_{\susp})$. Moreover, for every
	$C\in\calP$,
	\[
		\int_0^{\ell(C)} \calJ_f(\psi_s z)\,\dd s = I_f(\gamma_C)
	\]
	for any $z\in C$. Consequently, $m_C(\calJ_f)=\val_f(\gamma_C)$.
\end{proposition}

\begin{proof}
	By \cref{lem:Phi-bound}, $\Phi_f/r$ is bounded and continuous. Hence $\calJ_f$ is bounded and continuous on the interior of each suspension column. Since $\rho(0)=\rho(1)=0$, its values tend to $0$ on both sides of each identification $[\omega,r(\omega)]=[\sigma\omega,0]$, so $\calJ_f$ is continuous on all of $X_{\susp}$.

	On a single suspension column, we have
	\[
		\int_0^{r(\omega)} \calJ_f([\omega,s])\,\dd s = \Phi_f(\omega).
	\]
	Decomposing $C$ into the columns over $O_C$ and using \cref{prop:BHI-suspension} gives
	\[
		\int_0^{\ell(C)} \calJ_f(\psi_s z)\,\dd s = \sum_{\omega\in O_C}\Phi_f(\omega) = I_f(\gamma_C).
	\]
	Dividing by $\ell(C)=\ell(\gamma_C)$ gives $m_C(\calJ_f) = I_f(\gamma_C)/\ell(\gamma_C) = \val_f(\gamma_C)$.
\end{proof}

\subsection{Space average of $\calJ_f$}

We identify its space average with the constant term $\alpha_f$ of $fE_2$.

\begin{proposition}\label{prop:Jf-mean}
	For every $f\in M_0^!(\Gamma)$,
	\[
		\int_{X_{\susp}} \calJ_f\,\dd m_{\susp} = \alpha_f.
	\]
\end{proposition}

\begin{proof}
	We first prove the kernel identity
	\begin{equation}\label{eq:ave-kernel}
		\sum_{a\geq1}\int_0^1\int_0^1
		\frac{B_a^{(0)}(z;u,y)+B_a^{(1)}(z;u,y)}{(a+u+y)^2}\,\dd u\,\dd y
		=-\frac{\pi i}{z}\qquad(z\in\calA).
	\end{equation}
	For fixed $a,m$, set $p=m+u$, $q=a-m+y$ for $B_a^{(1)}$, and $p=(m+u)^{-1}$, $q=(a-m+y)^{-1}$ for $B_a^{(0)}$. The resulting rectangles partition the displayed regions up to their boundaries, and the changes of variables give
	\begin{align*}
		\sum_{a\geq1}\int_0^1\int_0^1\frac{B_a^{(1)}(z;u,y)}{(a+u+y)^2}\,\dd u\,\dd y
		&=\int_{\substack{p>1\\q>0}}\left(\frac1{z-p}-\frac1{z+q}\right)\frac{1}{(p+q)^2}\,\dd p\,\dd q,\\
		\sum_{a\geq1}\int_0^1\int_0^1\frac{B_a^{(0)}(z;u,y)}{(a+u+y)^2}\,\dd u\,\dd y
		&=\int_{\substack{0<p<1\\q>0}}\left(\frac1{z-p}-\frac1{z+q}\right)\frac{1}{(p+q)^2}\,\dd p\,\dd q.
	\end{align*}
	Combining them and then setting $q=tp$ gives
	\begin{align*}
		\int_0^\infty\int_0^\infty\frac{1}{(z-p)(z+q)(p+q)}\,\dd p\,\dd q
		&=\int_0^\infty\frac{1}{1+t}\,\dd t\int_0^\infty\frac{1}{(z-p)(z+tp)}\,\dd p\\
		&=\frac1z\int_0^\infty\frac{\log t-\pi i}{(1+t)^2}\,\dd t=-\frac{\pi i}{z},
	\end{align*}
	where $\log$ is the principal branch, proving \eqref{eq:ave-kernel}.

	By \cref{def:m-susp} and \cref{def:Jf},
	\[
		\int_{X_{\susp}}\calJ_f\,\dd m_{\susp} = \frac{1}{\overline{r}} \int_\Sigma\Phi_f \,\dd m_\Sigma.
	\]
	Under $x=1/(a+u)$, each parity sheet has
	\[
		\dd m_\Sigma = \frac{1}{2\log 2} \frac{\dd u\,\dd y}{(a+u+y)^2}.
	\]
	The uniform estimate \eqref{eq:uniform-estimate}, together with \cref{lem:roof-function}, justifies Fubini's theorem. Hence \eqref{eq:ave-kernel} gives
	\begin{align*}
		\int_{X_{\susp}}\calJ_f\,\dd m_{\susp} = \frac{1}{\overline{r}}\frac{1}{2\log 2} \int_{\calA} f(z)\left(-\frac{\pi i}{z}\right)\,\dd z
			= \frac{3}{\pi} \int_{\pi/3}^{2\pi/3}f(e^{it})\,\dd t.
	\end{align*}
	As shown in Kaneko--Zagier~\cite[Proposition~3]{KanekoZagier1998}, this quantity is precisely Atkin's inner product $\alpha_f=(f,1)_{\mathrm{At}}$. For completeness, we include the short argument.
	
	Since $f(-1/z)=f(z)$ and the quasimodular transformation law for $E_2$~\cite[Chapter~VII]{Serre1973} gives
	\[
		E_2(-1/z)=z^2E_2(z)+\frac{6z}{\pi i},
	\]
	the orientation-reversing substitution $z\mapsto-1/z$ yields
	\[
		2\int_{\calA}f(z)E_2(z)\,\dd z
		=-\frac6{\pi i}\int_{\calA}\frac{f(z)}{z}\,\dd z
		=-\frac6\pi\int_{\pi/3}^{2\pi/3}f(e^{it})\,\dd t.
	\]
	Cauchy's theorem on the region bounded by the oriented arc $\calA$ and the horizontal segment joining its endpoints, together with the Fourier expansion of $fE_2$, identifies the left-hand side as $-2\alpha_f$. Comparison with the preceding formula proves the claim.
\end{proof}

Combining \cref{thm:pollicott-equidistribution} with \cref{prop:Jf,prop:Jf-mean}, and using \cref{prop:closed-orbit-dictionary,lem:length-preserving}, we obtain
\[
	\lim_{L \to \infty} \frac{1}{\#\Pi(L)}\sum_{[\gamma]\in\Pi(L)}\val_f(\gamma) = \alpha_f.
\]
Thus the average of $\val_f(\gamma)$ tends to $\alpha_f$. In the next section, we combine this equidistribution with the ergodicity of $(\psi_s)_{s\in\R}$ to prove the stronger concentration statement in \cref{thm:main1}.

\section{Proof of \texorpdfstring{\cref{thm:main1}}{Theorem \ref{thm:main1}}}\label{section:proof-1}

Fix $f\in M_0^!(\Gamma)$ and put $g_f=\calJ_f-\alpha_f$. By \cref{prop:Jf,prop:Jf-mean}, $g_f\in C_b(X_{\susp})$ and
\[
	\int_{X_{\susp}}g_f\,\dd m_{\susp}=0.
\]
For $R>0$, put
\[
	(A_Rg_f)(z) \coloneqq \frac{1}{R}\int_0^R g_f(\psi_s z)\,\dd s.
\]
Let $\calE$ be the invariant $\sigma$-algebra, consisting of the measurable $A$ for which $m_{\susp}(\psi_s^{-1}A\mathbin\triangle A)=0$ for every $s \in\R$. Einsiedler--Ward~\cite[Corollary~8.15]{EinsiedlerWard2011}, applied to the real and imaginary parts of $g_f$, gives convergence in $L^1$ to $\mathbb{E}(g_f\mid\calE)$. Ergodicity of $(\psi_s)_{s\in\R}$ (\cref{lem:suspension-ergodic}) makes this conditional expectation constant almost everywhere, and its defining integral identity~\cite[Theorem~5.1]{EinsiedlerWard2011} identifies the constant with $\int_{X_{\susp}} g_f\,\dd m_{\susp}=0$. Hence
\[
	A_Rg_f\longrightarrow0\qquad\text{in }L^1(X_{\susp},m_{\susp}).
\]

For fixed $R$, continuity of the flow gives $A_Rg_f\in C_b(X_{\susp})$, while flow invariance of $m_C$ gives $m_C(g_f)=m_C(A_Rg_f)$ and hence
\[
	|m_C(g_f)| \leq m_C(|A_Rg_f|).
\]
Applying \cref{thm:pollicott-equidistribution} and first letting $L \to\infty$ yields
\begin{align*}
	\limsup_{L \to\infty} \frac{1}{\#\calP(L)} \sum_{C\in\calP(L)}|m_C(g_f)| \leq \int_{X_{\susp}}|A_Rg_f|\,\dd m_{\susp} = \|A_Rg_f\|_{L^1(X_{\susp}, m_{\susp})}.
\end{align*}
Now letting $R\to\infty$ gives
\[
	\lim_{L \to \infty} \frac{1}{\#\calP(L)} \sum_{C\in\calP(L)}|m_C(g_f)| = 0.
\]

By \cref{prop:closed-orbit-dictionary,lem:length-preserving,prop:Jf}, the correspondence $C \leftrightarrow [\gamma_C]$ gives $\#\calP(L) = \#\Pi(L)$ and $m_C(g_f) = \val_f(\gamma_C)-\alpha_f$. Therefore
\[
	\lim_{L \to \infty} \frac{1}{\#\Pi(L)} \sum_{[\gamma]\in\Pi(L)} |\val_f(\gamma)-\alpha_f| = 0.
\]
Finally, for every $\varepsilon > 0$,
\[
	\frac{\#\{[\gamma]\in\Pi(L): |\val_f(\gamma)-\alpha_f| \geq \varepsilon\}}{\#\Pi(L)} \leq \frac{1}{\varepsilon\#\Pi(L)} \sum_{[\gamma]\in\Pi(L)} |\val_f(\gamma)-\alpha_f|,
\]
and the right-hand side tends to $0$. Taking complements gives the assertion of \cref{thm:main1}.

\section{Concentration of nonreal val values}\label{section:proof-2}

We first verify that the formulation of the Kaneko--Shigeki conjecture given in the introduction agrees with their original formulation in terms of real quadratic irrationalities. Put
\[
	W \coloneqq \pmat{1 & 0 \\ 0 & -1} \in \PGL_2(\Z).
\]
Thus, for $\gamma = \smat{a & b \\ c & d} \in \Gamma$, we have $\gamma^\ast = W^{-1} \gamma W = \smat{a & -b \\ -c & d}$. Let $w_\gamma$ and $\widetilde{w}_\gamma$ denote the attracting and repelling fixed points of $\gamma$, respectively. For real quadratic irrationalities $u$ and $v$, we write $u\sim_\Gamma v$ if they are $\Gamma$-equivalent.

\begin{lemma}\label{lem:quadratic-matrix-equivalence}
	Let $\gamma \in \Gamma$ be hyperbolic. Then
	\begin{align}
		[\gamma]=[\gamma^\ast] &\quad \Longleftrightarrow \quad w_\gamma\sim_\Gamma -w_\gamma, \label{eq:star-fixed-equivalence}\\
		[\gamma]=[(\gamma^\ast)^{-1}] &\quad \Longleftrightarrow \quad w_\gamma\sim_\Gamma -\widetilde{w}_\gamma. \label{eq:star-inverse-fixed-equivalence}
	\end{align}
\end{lemma}

\begin{proof}
	Since $W$ acts on $\mathbb{P}^1(\R)$ by $x\mapsto -x$, the ordered pairs of attracting and repelling fixed points of $\gamma^\ast$ and $(\gamma^\ast)^{-1}$ are $(-w_\gamma,-\widetilde{w}_\gamma)$ and $(-\widetilde{w}_\gamma,-w_\gamma)$, respectively.

	If $\delta^{-1} \gamma \delta = \gamma^\ast$, comparison of attracting fixed points gives $\delta^{-1} w_\gamma = -w_\gamma$. Conversely, suppose that $\delta^{-1} w_\gamma = -w_\gamma$. Galois conjugation then gives $\delta^{-1}\widetilde{w}_\gamma = -\widetilde{w}_\gamma$. Choose a lift $\widetilde{\gamma}\in\SL_2(\R)$ with eigenvalues $\lambda,\lambda^{-1}$, where $\lambda>1$, and a lift $\widetilde{\delta}$ of $\delta$. The two conjugates $\widetilde{\delta}^{-1}\widetilde{\gamma}\widetilde{\delta}$ and $W^{-1}\widetilde{\gamma}W$ have the same $\lambda$- and $\lambda^{-1}$-eigenspaces by the preceding identities, and act on them by the same eigenvalues. Hence they agree on a basis of $\R^2$, so $\delta^{-1}\gamma\delta=\gamma^\ast$. This proves \eqref{eq:star-fixed-equivalence}.

	The same argument proves \eqref{eq:star-inverse-fixed-equivalence} after interchanging the two fixed points.
\end{proof}

In what follows, put $w=w_\gamma$, $\widetilde{w}=\widetilde{w}_\gamma$, and $\Lambda_w =\Z w+\Z$. Define
\[
	\mathcal O_w
	\coloneqq \{\alpha\in \Q(w): \alpha \Lambda_w \subseteq \Lambda_w\}.
\]
This is the multiplier order of $\Lambda_w$. As recalled in the introduction, Kaneko's original definition agrees with our cycle-integral normalization, so that $\val(w)=\val(\gamma)$.

We first record the equivalence
\[
    w\sim_\Gamma-w
    \quad\Longleftrightarrow\quad
    \mathcal O_w^\times\text{ contains a unit of norm }-1.
\]
Indeed, if $\alpha\in\mathcal O_w^\times$ has norm $-1$, write
\[
	\alpha w=aw+b,\qquad \alpha=cw+d.
\]
On the basis $(w,1)$, multiplication by $\alpha$ has matrix $\smat{a&c\\b&d}$, so $ad-bc=-1$. Hence
\[
	\delta=\pmat{-a&-b\\c&d}\in\SL_2(\Z),
	\qquad \delta w=\frac{-aw-b}{cw+d}=-w.
\]
Conversely, if $\delta=\smat{A&B\\C&D}\in\SL_2(\Z)$ satisfies $\delta w=-w$, put $\alpha=Cw+D$. Then $\alpha w=-Aw-B$, so multiplication by $\alpha$ on $\Lambda_w$ has matrix $\smat{-A&C\\-B&D}$, whose determinant is $-AD+BC=-1$. Thus it is an automorphism of $\Lambda_w$, and $\alpha\in\mathcal O_w^\times$ has norm $-1$.

Kaneko and Shigeki~\cite{KanekoShigeki2010} distinguish two cases according to the norm of the fundamental unit of $\mathcal O_w$. If it has norm $-1$, the equivalence above gives $w\sim_\Gamma-w$. By \cref{lem:quadratic-matrix-equivalence} and the symmetry \eqref{eq:Kaneko-symmetry}, this implies that $\val(w)$ is real. If the fundamental unit has norm $1$, they conjecture that
\[
	\val(w)\in\R \quad\Longleftrightarrow\quad w\sim_\Gamma-\widetilde w.
\]
Thus their conjecture may be written uniformly as
\[
	\val(w)\in\R \quad\Longleftrightarrow\quad w\sim_\Gamma -w \ \text{or}\ w\sim_\Gamma -\widetilde{w}.
\]
By \cref{lem:quadratic-matrix-equivalence}, this is precisely \cref{conj:Kaneko-Shigeki}.

We next estimate the number of classes that must be removed. Following Sarnak~\cite{Sarnak2007}, define involutions on $\Pi$ by
\[
	\phi_R([\gamma])\coloneqq[\gamma^{-1}], \qquad \phi_W([\gamma])\coloneqq[\gamma^\ast], \qquad \phi_A\coloneqq\phi_R \circ \phi_W.
\]
Thus $\phi_A([\gamma])=[(\gamma^\ast)^{-1}]$.

\begin{proposition}\label{prop:real-density-zero-sarnak}
	Assume \cref{conj:Kaneko-Shigeki}. Then
	\[
		\lim_{L \to \infty} \frac{\#\{[\gamma]\in\Pi(L):\val(\gamma)\in\R\}}{\#\Pi(L)} = 0.
	\]
	Consequently, $\#\Pi_{\mathrm{nr}}(L) \sim \#\Pi(L)$.
\end{proposition}

\begin{proof}
	For $[\gamma]\in\Pi$, put $t(\gamma)\coloneqq \left|\tr\widetilde{\gamma}\right|$, where $\widetilde{\gamma}\in\SL_2(\Z)$ is either lift of $\gamma$. Let $N(X)$ denote the number of classes in $\Pi$ with $t(\gamma)\leq X$, and let $N_W(X)$ and $N_A(X)$ denote the numbers among them fixed by $\phi_W$ and $\phi_A$, respectively. Sarnak~\cite[Theorem~2, (11), (12), and~(14)]{Sarnak2007} proves
	\[
		N(X)\sim\frac{X^2}{2\log X}, \qquad N_W(X)\sim\frac{X}{2\log X},\qquad N_A(X)\ll X(\log X)^2.
	\]

	Put $X_L \coloneqq2\cosh(L/2)$. Since
	\[
		t(\gamma)=2\cosh\frac{\ell(C_\gamma)}{2},
	\]
	we have $\#\Pi(L)=N(X_L)$. By \cref{conj:Kaneko-Shigeki}, every class in $\Pi(L)$ for which $\val(\gamma)$ is real is fixed by either $\phi_W$ or $\phi_A$. Hence
	\[
		\frac{\#\{[\gamma]\in\Pi(L):\val(\gamma)\in\R\}}{\#\Pi(L)} \leq \frac{N_W(X_L)+N_A(X_L)}{N(X_L)} \to 0.
	\]
	The final assertion follows immediately by taking complements in $\Pi(L)$.
\end{proof}

Note that Sarnak gives
\[
	N_A(X)\sim \frac{97}{8\pi^2}X(\log X)^2,
\]
but Parkkonen--Paulin~\cite[Theorem~14]{ParkkonenPaulin2024pre} point out that Sarnak's argument contains an error. After converting from length to trace, their result gives instead
\[
	N_A(X)\sim \frac{3}{\pi^2}X(\log X)^2.
\]

\begin{proof}[Proof of \cref{thm:main2}]
	Fix $\varepsilon>0$. By \cref{thm:main1,prop:real-density-zero-sarnak},
	\[
		\frac{\#\{[\gamma]\in\Pi_{\mathrm{nr}}(L) : |\val(\gamma)-720|\geq\varepsilon\}}{\#\Pi_{\mathrm{nr}}(L)} \leq \frac{\#\{[\gamma]\in\Pi(L) : |\val(\gamma)-720|\geq\varepsilon\}}{\#\Pi(L)} \frac{\#\Pi(L)}{\#\Pi_{\mathrm{nr}}(L)} \to 0.
	\]
	Taking complements proves \cref{thm:main2}.
\end{proof}

\appendix
\crefalias{section}{appendix}
\section{Proof of \texorpdfstring{\cref{thm:pollicott-equidistribution}}{Theorem~\ref{thm:pollicott-equidistribution}}}
\label{sec:Proof-Pollicott}

As explained in \cref{rem:pollicott-route}, \cref{thm:pollicott-equidistribution} can be proved by adapting Pollicott's transfer-operator argument to the continued-fraction suspension used here. In this appendix we give a different proof, which is more arithmetic in flavor. Its main input is Ustinov's effective counting theorem for reduced quadratic irrationals~\cite{Ustinov2013}.

We retain the notation of \cref{section:suspension}. Thus, for $C\in\calP$, the set $O_C\subset\Sigma$ is the corresponding primitive $\sigma$-orbit, and
\[
	\ell(C)=\sum_{\omega\in O_C}r(\omega),
	\qquad
	\overline{r}=\int_\Sigma r\,\dd m_\Sigma =\frac{\pi^2}{6\log 2}.
\]
For a function $\Phi$ on $\Sigma$, put
\[
	S_C(\Phi)\coloneqq\sum_{\omega\in O_C}\Phi(\omega).
\]

\subsection{Counting periodic points in cylinders}

Let $\mathcal R_{\mathrm{red}}$ be the set of reduced quadratic irrationals $w\in I$, equivalently the periodic points of the Gauss map $T$. For $w\in\mathcal R_{\mathrm{red}}$, let $\widetilde{w}$ be its Galois conjugate, let $N=N(w)$ be the least positive even integer such that $T^N w = w$, and write $w=[\overline{a_1,\ldots,a_N}]$. It is known that $-1/\widetilde{w} =[\overline{a_N,\ldots,a_1}]$. Hence, for $\epsilon\in\Z/2\Z$,
\[
	\omega_{w,\epsilon}
	\coloneqq
	\left(w,-\frac{1}{\widetilde{w}},\epsilon\right) = \left([\overline{a_1,\ldots,a_N}], [\overline{a_N,\ldots,a_1}], \epsilon \right) \in \Sigma
\]
has least $\sigma$-period $N$ and determines a primitive suspension orbit $C_{w,\epsilon}\in\calP$. By \eqref{eq:periodic-roof-sum} and \cref{prop:closed-orbit-dictionary,lem:length-preserving},
\begin{align}\label{eq:varrho-suspension}
	\varrho(w) \coloneqq \ell(C_{w,\epsilon}) = -2\sum_{j=0}^{N-1}\log T^j w.
\end{align}
This is independent of $\epsilon$ and agrees with the length used by Ustinov. With this convention, Ustinov's theorem takes the following slightly reformulated form, where we express the cutoff in terms of $\varrho(w)$ and write the main term as a double integral.

\begin{theorem}[{Ustinov~\cite[Theorem~3]{Ustinov2013}}]\label{thm:ubs-count}
	For fixed $\alpha,\beta\in[0,1]$ and every $\varepsilon > 0$, as $L\to\infty$,
	\begin{align*}
		\#\left\{
		w\in\mathcal R_{\mathrm{red}}:\varrho(w)\leq L,\ 0\leq w\leq\alpha,\ 0\leq-\frac{1}{\widetilde{w}}\leq\beta \right\} =
		\frac{e^L}{2\zeta(2)}
		\int_0^\alpha\int_0^\beta
		\frac{\dd u\,\dd v}{(1+uv)^2}
		+O_\varepsilon(e^{(3/4+\varepsilon)L}).
	\end{align*}
\end{theorem}

For $\omega=(x,y,\epsilon_\omega)\in\Sigma$, write
\[
	x=[a_1(\omega),a_2(\omega),\ldots], \qquad
	y=[a_0(\omega),a_{-1}(\omega),\ldots].
\]
By a \emph{cylinder} we mean a set obtained by prescribing the parity and finitely many consecutive digits $a_j(\omega)$. For $\epsilon\in\Z/2\Z$ and a finite word
$\mathbf b=(b_1,\ldots,b_q)\in\N^q$, put
\begin{align*}
	I_{\mathbf b} &\coloneqq \left\{w=[a_1,a_2,\ldots]\in I: a_j=b_j\ \text{for }1\leq j\leq q \right\},\\
	A(\epsilon;\mathbf b) &\coloneqq \left\{\omega\in\Sigma: \epsilon_\omega=\epsilon,\ a_j(\omega)=b_j\ \text{for }1\leq j\leq q \right\}.
\end{align*}
The set $I_{\mathbf b}$ is the intersection of $I$ with an interval having rational endpoints, so no element of $\mathcal R_{\mathrm{red}}$ lies on its boundary.

\begin{lemma}\label{lem:marked-cylinder-count}
	For every $\epsilon\in\Z/2\Z$, every finite word $\mathbf b$, and every $L>0$,
	\begin{align}\label{eq:exact-cylinder-count}
		\sum_{C\in\calP(L)} S_C\bigl(\mathbf{1}_{A(\epsilon;\mathbf b)}\bigr)
		= \#\left\{w\in\mathcal R_{\mathrm{red}}: \varrho(w)\leq L,\ w\in I_{\mathbf b} \right\}.
	\end{align}
\end{lemma}

\begin{proof}
	For fixed $\epsilon$, the map $w\mapsto\omega_{w,\epsilon}$ is a bijection from $\mathcal R_{\mathrm{red}}$ onto the periodic points of $\Sigma$ whose parity coordinate is $\epsilon$. By \eqref{eq:varrho-suspension}, its associated primitive orbit has length $\varrho(w)$, and
	\[
		\omega_{w,\epsilon}\in A(\epsilon;\mathbf b) \quad\Longleftrightarrow\quad w\in I_{\mathbf b}.
	\]
	Since every periodic point of $\Sigma$ lies on a unique primitive periodic $\sigma$-orbit, the left-hand side of \eqref{eq:exact-cylinder-count} is the number of periodic points $\omega\in A(\epsilon;\mathbf b)$ whose primitive orbit has length at most $L$. Under the above bijection $w\mapsto\omega_{w,\epsilon}$, these are exactly the $w\in\mathcal R_{\mathrm{red}}$ such that $\varrho(w)\leq L$ and $w\in I_{\mathbf b}$.
\end{proof}

Fix $0 < \varepsilon < 1/4$. Since $-1/\widetilde{w}\in I$ for $w\in\mathcal R_{\mathrm{red}}$, taking $\beta=1$ in \cref{thm:ubs-count} and subtracting at the endpoints of $I_{\mathbf b}$ gives
\[
	\sum_{C\in\calP(L)} S_C\bigl(\mathbf{1}_{A(\epsilon;\mathbf b)}\bigr)
	= \frac{e^L}{2\zeta(2)} \int_{I_{\mathbf b}\times I} \frac{\dd u\,\dd v}{(1+uv)^2} +O_{\varepsilon}\!\left(e^{(3/4+\varepsilon)L}\right).
\]
By the definition of $m_\Sigma$,
\[
	m_\Sigma\bigl(A(\epsilon;\mathbf b)\bigr) = \frac{1}{2\log 2} \int_{I_{\mathbf b}\times I} \frac{\dd u\,\dd v}{(1+uv)^2}.
\]
Hence, since $\overline{r} =\zeta(2)/\log 2$,
\begin{align}\label{eq:cylinder-count-asymptotic}
	\sum_{C\in\calP(L)} S_C\bigl(\mathbf{1}_{A(\epsilon;\mathbf b)}\bigr)
	= \frac{m_\Sigma\bigl(A(\epsilon;\mathbf b)\bigr)}{\overline{r}}e^L +o(e^L).
\end{align}
Every cylinder is a translate under some power of $\sigma$ of a set of the form $A(\epsilon;\mathbf b)$. Since $\sigma$ permutes each $O_C$ and preserves $m_\Sigma$, \eqref{eq:cylinder-count-asymptotic} holds for every cylinder $A$.

We next pass to inverse-length weighted sums by Stieltjes partial summation. For a cylinder $A$, put
\[
	B_A(L) \coloneqq \sum_{C\in\calP(L)}S_C(\mathbf{1}_A).
\]
Since every $\sigma$-period is even and $r>\log 2$, we have $\ell(C)>2\log 2$ for every $C\in\calP$. Hence, for any $0<\ell_0<2\log 2$,
\begin{align}\label{eq:inverse-length-cylinder-count}
	\sum_{C\in\calP(L)} \frac{S_C(\mathbf{1}_A)}{\ell(C)} = \frac{B_A(L)}{L} +\int_{\ell_0}^L\frac{B_A(t)}{t^2}\,\dd t = \frac{m_\Sigma(A)}{\overline{r}}\frac{e^L}{L} +o\left(\frac{e^L}{L}\right).
\end{align}

\subsection{Convergence on the base}

For $C\in\calP$, define
\[
	\nu_C\coloneqq\frac{1}{\ell(C)}\sum_{\omega\in O_C}\delta_\omega,
	\qquad
	\nu_L\coloneqq\frac{1}{\#\calP(L)}\sum_{C\in\calP(L)}\nu_C,
	\qquad
	\nu\coloneqq\frac{1}{\overline{r}}m_\Sigma,
\]
where $\delta_\omega$ denotes the Dirac measure at $\omega$, and $\nu_L$ is considered for sufficiently large $L$. By \cref{prop:closed-orbit-dictionary,lem:length-preserving}, the prime geodesic theorem gives $\#\calP(L)\sim e^L/L$. Hence \eqref{eq:inverse-length-cylinder-count} implies
\begin{align}\label{eq:cylinder-convergence}
	\nu_L(A)\longrightarrow\nu(A)
\end{align}
for every cylinder $A$.

These measures are $\sigma$-invariant, and
\begin{align}\label{eq:roof-first-moment}
	\int_\Sigma r\,\dd\nu_C = \int_\Sigma r\,\dd\nu_L = \int_\Sigma r\,\dd\nu =1.
\end{align}
Indeed, the first integral equals $1$ by $\ell(C)=\sum_{\omega\in O_C}r(\omega)$, the second by averaging, and the third by the definition of $\overline{r}$. Since $r>\log 2$, we also have $\nu_L(\Sigma),\nu(\Sigma)\leq1/\log 2$.

For $j\in\Z$ and an integer $B\geq2$, set
\[
	E_j(B)\coloneqq\{\omega\in\Sigma:a_j(\omega)\geq B\}.
\]
Since $a_1(\sigma^{j-1}\omega)=a_j(\omega)$, \cref{lem:roof-function} gives $r(\sigma^{j-1}\omega)\geq\log B$ on $E_j(B)$. Thus $\sigma$-invariance and \eqref{eq:roof-first-moment} yield
\[
	\nu_L(E_j(B))\leq\frac{1}{\log B},
	\qquad
	\nu(E_j(B))\leq\frac{1}{\log B}.
\]

Enumerate $\Z=\{j_1,j_2,\ldots\}$ and, for $\delta>0$, choose integers $B_m\geq2$ such that $\sum_{m\geq1}(\log B_m)^{-1}<\delta$. Then
\[
	K_\delta \coloneqq \{\omega\in\Sigma:a_{j_m}(\omega)<B_m\text{ for every }m\geq1\}
\]
is compact, since the continued-fraction coordinate map identifies $\Sigma$ homeomorphically with $\N^\Z\times\Z/2\Z$, and $K_\delta$ with a product of finite discrete sets. Moreover, $K_\delta^c\subset\bigcup_{m\geq1}E_{j_m}(B_m)$, so
\begin{align}\label{eq:compact-core}
	\nu_L(K_\delta^c)<\delta,
	\qquad
	\nu(K_\delta^c)<\delta.
\end{align}

Let $\Phi\in C_b(\Sigma)$ and $\eta>0$. Since the cylinders form a clopen basis of $\Sigma$, for every $\omega\in K_\delta$ there exists a cylinder $U_\omega$ containing $\omega$ such that
\[
	|\Phi(\omega')-\Phi(\omega'')|<\eta
	\qquad
	(\omega',\omega''\in U_\omega\cap K_\delta).
\]
By compactness, finitely many such cylinders $U_1,\ldots, U_m$ cover $K_\delta$. Prescribe simultaneously the parity coordinate and all digit coordinates occurring in the definitions of $U_1,\ldots, U_m$. Since each of these digit coordinates takes only finitely many values on $K_\delta$, this yields finitely many pairwise disjoint cylinders $V_1,\ldots,V_s$ that intersect $K_\delta$ and cover it. Moreover, each $V_i$ is contained in some $U_k$.

Choose $\omega_i\in V_i\cap K_\delta$ and put
\[
	\Psi \coloneqq \sum_{i=1}^s \Phi(\omega_i)\mathbf{1}_{V_i}.
\]
Then $\Psi\in C_b(\Sigma)$ and
\[
	\|\Psi\|_\infty\leq\|\Phi\|_\infty,
	\qquad
	|\Phi-\Psi|<\eta \quad\text{on }K_\delta.
\]
For $\mu=\nu_L$ or $\nu$, \eqref{eq:compact-core} gives
\[
	\left|\int_\Sigma(\Phi-\Psi)\,\dd\mu\right|
	\leq
	\frac{\eta}{\log 2}+2\|\Phi\|_\infty\delta.
\]
Applying \eqref{eq:cylinder-convergence} to $\Psi$ and then letting $\eta,\delta\to0$, we obtain
\begin{align}\label{eq:base-weak-convergence}
	\int_\Sigma\Phi\,\dd\nu_L
	\longrightarrow
	\int_\Sigma\Phi\,\dd\nu
\end{align}
for every $\Phi \in C_b(\Sigma)$.

\subsection{Passage to the suspension}

For $F\in C_b(X_{\susp})$, define $F^\sharp \colon \Sigma \to \C$ by
\[
	F^\sharp(\omega) \coloneqq \int_0^{r(\omega)}F([\omega,t])\,\dd t.
\]
Then $F^\sharp$ is continuous and
\begin{equation}\label{eq:column-bound}
	|F^\sharp(\omega)|\leq\|F\|_\infty r(\omega).
\end{equation}

For $M>0$, define
\[
	r_M(\omega)\coloneqq\min\{r(\omega),M\},
	\qquad
	q_M(\omega)\coloneqq r(\omega) - r_M(\omega) = \max\{r(\omega) - M, 0\}.
\]
Since $r_M\in C_b(\Sigma)$, \eqref{eq:base-weak-convergence} and \eqref{eq:roof-first-moment} give
\[
	\int_\Sigma q_M\,\dd\nu_L =1-\int_\Sigma r_M\,\dd\nu_L
	\longrightarrow
	1-\int_\Sigma r_M\,\dd\nu =\int_\Sigma q_M\,\dd\nu.
\]
Set $H_M \coloneqq \{\omega\in\Sigma:r(\omega)>M\}$. Because $r\leq2q_{M/2}$ on $H_M$,
\[
	\limsup_{L\to\infty}\int_{H_M}r\,\dd\nu_L
	\leq2\int_\Sigma q_{M/2}\,\dd\nu.
\]
Since $r\in L^1(\nu)$, the right-hand side tends to $0$ as $M\to\infty$. Thus
\begin{align}\label{eq:roof-tail}
	\lim_{M\to\infty}\limsup_{L\to\infty}
	\int_{H_M}r\,\dd\nu_L=0.
\end{align}

Choose a continuous $\chi_M\colon[0,\infty)\to[0,1]$ equal to $1$ on $[0,M]$ and to $0$ on $[2M,\infty)$. By \eqref{eq:column-bound}, the function $G_M(\omega)\coloneqq F^\sharp(\omega)\chi_M(r(\omega))$ belongs to $C_b(\Sigma)$, so \eqref{eq:base-weak-convergence} applies to $G_M$. Writing
\[
	F^\sharp = G_M+ F^\sharp(1-\chi_M(r)),
\]
the first term is handled by \eqref{eq:base-weak-convergence}, while $|F^\sharp(1-\chi_M(r))| \leq \|F\|_\infty r\,\mathbf1_{H_M}$ by \eqref{eq:column-bound}. Hence
\begin{align*}
	\limsup_{L\to\infty} \left|\int_\Sigma F^\sharp\,\dd\nu_L - \int_\Sigma F^\sharp\,\dd\nu \right|
	&\leq \|F\|_\infty\limsup_{L\to\infty}\int_{H_M}r\,\dd\nu_L + \|F\|_\infty\int_{H_M}r\,\dd\nu.
\end{align*}
Letting $M\to\infty$ and using \eqref{eq:roof-tail}, we conclude that
\begin{align}\label{eq:column-convergence}
	\int_\Sigma F^\sharp \,\dd\nu_L
	\longrightarrow
	\int_\Sigma F^\sharp\,\dd\nu.
\end{align}

Finally,
\[
	\int_\Sigma F^\sharp \,\dd\nu_L = \frac{1}{\#\calP(L)}\sum_{C\in\calP(L)}m_C(F),
\]
while
\[
	\int_\Sigma F^\sharp\,\dd\nu = \frac{1}{\overline{r}} \int_\Sigma\int_0^{r(\omega)}F([\omega,t])\,\dd t\,\dd m_\Sigma(\omega) = \int_{X_{\susp}}F\,\dd m_{\susp}.
\]
Thus \eqref{eq:column-convergence} is precisely \cref{thm:pollicott-equidistribution}.

\subsection*{Acknowledgements}

The author would like to express his sincere gratitude to Masanobu Kaneko, who first introduced him to this problem during his undergraduate years and has provided many valuable insights and suggestions over the years. The author also thanks Paloma Bengoechea and \"{O}zlem Imamo\={g}lu for helpful exchanges on cycle integrals dating back to 2018. This work was supported by JSPS KAKENHI Grant Number JP24K16901. 

\subsection*{AI \& computational resource disclosure}

ChatGPT (GPT-5.6) and OpenAI Codex were used in the preparation of this paper for discussing and refining arguments, locating relevant literature, assisting with the development of Mathematica code, and supporting proof auditing, reference verification, and manuscript revision. All mathematical statements and proofs were independently verified by the author.

\bibliographystyle{amsalpha}
\bibliography{References} 

\end{document}